\documentclass[]{article}
\usepackage{amsmath}
\usepackage{amsfonts}
\usepackage{amssymb}
\usepackage{amsthm}
\usepackage{graphicx}
\usepackage{float}
\usepackage{enumerate}
\usepackage{comment}
\usepackage[utf8]{inputenc}

\usepackage[
backend=biber,
style=ieee,
sorting=nyt,
dashed=false,
]{biblatex}
\newtheorem{theorem}{Theorem}[section]
\newtheorem{lemma}[theorem]{Lemma}
\newtheorem{Remark}[theorem]{Remark}
\newtheorem{corollary}[theorem]{Corollary}
\newtheorem{Definition}[theorem]{Definition}
\newtheorem{Proposition}[theorem]{Proposition}
\title{On a class of fully nonlinear elliptic equation in dimension two}
\author{Filomena Pacella \thanks{pacella@mat.uniroma1.it} \: and \: David Stolnicki \thanks{david.stolnicki@uniroma1.it}}
\date{Dipartimento di Matematica, Università di Roma Sapienza, Italy}
\begin{document}

\maketitle

\noindent\textbf{Abstract:} We study existence and asymptotic behavior of radial positive solutions of some fully nonlinear equations involving Pucci's extremal operators in dimension two. In particular we prove the existence of a positive solution of a fully nonlinear version of the Liouville equation in the plane. Moreover for the $\mathcal{M}^- _{\lambda,\Lambda}$ operator, we show the existence of a critical exponent and give bounds for it. 
\section{Introduction}

\par In this paper we study positive radial solutions for the fully nonlinear Lane Emden equation driven by Pucci's extremal operators: \\

\begin{equation} \label{genprob}
		-\mathcal{M}_{\lambda,\Lambda}^\pm (D^2 u) = u^p \qquad \text{ in } \Omega\\
\end{equation} \\

\noindent where $\Omega \subset \mathbb{R}^2$ is either the ball $B_R$ of radius $R >0$, centered at the origin or the whole plane $\mathbb{R}^2$. In the case of the ball we will impose the homogeneous Dirichlet boundary condition. \\
We recall that for a $\mathcal{C}^2$ function in $\mathbb{R}^N$, $N \geq 2,$ the Pucci's operators are defined by:

\begin{align*}
	\mathcal{M}_{\lambda,\Lambda}^+ (D^2 u) = \Lambda \sum_{\mu_i >0} \mu_i + \lambda \sum_{\mu_i <0} \mu_i \\
	\mathcal{M}_{\lambda,\Lambda}^- (D^2 u) = \lambda \sum_{\mu_i >0} \mu_i + \Lambda \sum_{\mu_i <0} \mu_i \\
\end{align*}  

\noindent where $0<\lambda\leq \Lambda$ are the ellipticity constants and $\mu_i=\mu_i(D^2 u)$ , $i=1,\ldots,N$ are the eigenvalues of the hessian matrix $D^2 u$. Associated to $\mathcal{M}_{\lambda,\Lambda}^\pm$ are dimension like numbers $\tilde{N}_{+}$ and $\tilde{N}_{-}$ defined as

\begin{equation}\label{Dimension}
	\tilde{N}_{+}= \frac{\lambda}{\Lambda}(N-1)+1 \qquad \text{ and } \tilde{N}_{-}= \frac{\Lambda}{\lambda}(N-1)+1
\end{equation}

These numbers allow to give estimates for the exponent $p$ for which existence or nonexistence of solutions of (\ref{genprob}) in $B_R$ or $\mathbb{R}^N$ holds when $N\geq 3$ and $\tilde{N}_+ > 2$, (note that $\tilde{N}_{-}$ is always larger than two if $N \geq 3$). 

Indeed a first result obtained in \cite{CL1} shows that if $ N \geq 3$ and $p \leq \frac{\tilde{N}_{\pm}}{\tilde{N}_{\pm} - 2}$ then no nontrivial positive viscosity supersolutions of (\ref{genprob}) exist in $\mathbb{R}^N$. Using this result, the existence of positive solutions in bounded domains, not necessarily radial, was proved in \cite{QS} for the same range of exponents. 

\par In the radial setting Felmer and Quaas in \cite{FQ} (see also \cite{FQ1}) provided, for $N\geq3$ , the existence of critical exponents $p_{+}^*$ for $\mathcal{M}_{\lambda,\Lambda}^+$ and $p_{-}^*$ for $\mathcal{M}_{\lambda,\Lambda}^-$ satisfying 

\begin{align}\label{ce}
	\max \left\{ \frac{\tilde{N}_+ }{\tilde{N}_+ - 2} , \frac{N+2}{N-2} \right\}  <\; &p_+^*  < \frac{\tilde{N}_+ + 2 }{\tilde{N}_+ - 2} \\ 
	 \frac{\tilde{N}_- + 2 }{\tilde{N}_- - 2}  <\; &p_-^*  < \frac{N+2}{N-2} \nonumber
\end{align}  

\noindent which are thresholds for the existence of radial positive solutions of (\ref{genprob}) in the ball or in $\mathbb{R}^N$.

\par More precisely they proved that existence in $B_R$ holds if and only if $ p < p_{\pm}^*$, while existence in $\mathbb{R}^N$ holds if and only if $ p  \geq p_{\pm}^*$. In the last case a complete classification of the solutions was given, according to the decay at infinity. Recently, the same kind of results has been obtained in \cite{MNP} for more general equations with an alternative proof.

\par In dimension N=2, the situation for the two operators $\mathcal{M}_{\lambda,\Lambda}^+,\mathcal{M}_{\lambda,\Lambda}^-$ is different. Indeed, by (\ref{Dimension}) we have that $\tilde{N}_{+} < 2 $, while $\tilde{N}_{-} > 2 $, for $\lambda < \Lambda$. In the first case the result of \cite{CL1} still implies the nonexistence of nontrivial solutions in $\mathbb{R}^N$ for the equation (\ref{genprob}) for $\mathcal{M}_{\lambda,\Lambda}^+$ for every $p \in (1,\infty)$.
\par As a consequence a positive radial solution of 

\begin{equation} 
	\begin{cases} 
		-\mathcal{M}_{\lambda,\Lambda}^+ D^2 u = u^p \qquad \text{ in } B_R\\
		u > 0 \qquad \text{ in } B_R \label{prob} \\
		u = 0 \qquad \text{ on } \partial B_R  
	\end{cases}
\end{equation} \\

\noindent exists for every $ 1 < p < \infty$ as in \cite{QS}. Thus no critical exponent as in (\ref{ce}) can be defined for $\mathcal{M}_{\lambda,\Lambda}^+$, in dimension 2.

\par Instead in the case of the operator $\mathcal{M}_{\lambda,\Lambda}^-$ the number $\tilde{N}_- -2$ is still positive which suggests that a critical exponent $p^*_-$ as in higher dimensions could exist, though, any upper bound for it as in (\ref{ce}), could not be given, since $N=2$.
\par In this paper we obtain new results for (\ref{genprob}) in $\mathbb{R}^2$, but of different kind for each one of the two operators.

In the case of $\mathcal{M}_{\lambda,\Lambda}^+$, when a unique radial positive solution of (\ref{prob}) in the ball exists for every $p > 1$( see Theorem \ref{prop}), we study the asymptotic behavior of $u_p$ as $p\rightarrow \infty$. This is done by analyzing the rescaled function

\begin{align}\label{rescale}
z_p = z_p(r) = \frac{p}{u_p(0)}(u(\varepsilon_p r) -u_p(0)),
\end{align} 

\noindent for $ r = \vert x \vert$, and

\begin{equation*}
	\varepsilon_p^{-2} = p \cdot u_p(0)^{p-1},
\end{equation*}

\noindent as in the semilinear case (\cite{AG},\cite{MIP}). As a byproduct, we obtain the existence of a radial solution of the "fully nonlinear Liouville equation" in the plane. As far as we know, this is the first existence result for this equation.

\noindent More precisely we get:
\begin{theorem} \label{conv} \label{description}
	The function $z_p$ converges, up to a subsequence, in $\mathcal{C}_{loc}^2(\mathbb{R}^2)$ to a radial function $z$ which satisfies:
	
	\begin{align}\label{liouv}
	-\mathcal{M}_{\lambda,\Lambda}^+ (D^2 z)(x) = e^z.
	\end{align}

\noindent Moreover $z$ is negative, radially decreasing and, as a function of $\vert x \vert$, changes concavity only once at $|x|=R_0=2\sqrt{2\lambda}$. Finally:
\begin{equation*}
	z(x)=\log\left( \frac{1}{(1+\frac{1}{8\lambda}|x|^2)^2} \right) \text{ in } B_{R_0}.
\end{equation*}
\end{theorem}

In the case of the operator $\mathcal{M}_{\lambda,\Lambda}^-$, we are able to prove that in dimension $N=2$, a critical exponent $p_-^*$ having similar features as the one for $N\geq3$ indeed exists and we provide bounds for it.
\par As mentioned before, to get an upper bound for $p_-^*$ is not obvious since the corresponding estimate in higher dimensions in (\ref{ce}) blows up when $N=2$.
\par We get it through the existence of another relevant exponent, denoted by $\tilde{p}_-$, which is responsible for the existence or lack of periodic orbits for a related dynamical system that we study, following the approach of \cite{MNP}.
\par As it will be made clear in Section 3, the periodic orbits of this auxiliary dynamical system are related to the nonexistence of solutions in the ball, and possibly allow the existence of entire oscillating radial solutions for (\ref{genprob}).
\par Referring to Definition 3.5 for fast, slow, or pseudo-slow decaying solutions we state the main result for the operator $\mathcal{M}_{\lambda,\Lambda}^-$.

\begin{theorem}[Critical Exponent]  \label{criticalexp}
	
	Let the dimension $N$ be two, there are exponents $p_-^*$ and $\tilde{p}_- $ satisfying:
	\begin{equation}\label{expbound}
	 \frac{\tilde{N}_{-} +2}{\tilde{N}_{-} -2} < p_{-}^* \leq  \tilde{p}_{-} \leq \frac{\tilde{N}_{-} +2}{\tilde{N}_{-} -2} + \frac{4}{\tilde{N}_{-} -2+\frac{\lambda}{\Lambda}(\tilde{N}_{-} -2)^2} 
	\end{equation}
 with $\tilde{N}_-$ as in (\ref{Dimension}), such that, considering equation (\ref{genprob}) for $\mathcal{M}_{\lambda,\Lambda}^-$: \\
	\begin{enumerate}[i)] 
	\item for $ p < p_{-}^*$ there is no nontrivial radial positive solution of problem (\ref{genprob}) in $\mathbb{R}^2$, while, for every $R>0$ there is a unique positive radial solution in $B_R$.
	\item if $ p = p_{-}^*$ there is a unique fast decaying radial positive solution of (\ref{genprob}) in $\mathbb{R}^2$
	\item for $ p_{-}^* < p < \tilde{p}_{-}$ there is a unique positive radial solution of (\ref{genprob}) in $\mathbb{R}^2$, which may be either pseudo-slow or slow decaying.
	\item for $ p > \tilde{p}_{-}$ there is a unique slow decaying solution of (\ref{genprob}) in $\mathbb{R}^2$ 
	\item for $ p > p_{-}^* $ there is no positive radial solution of (\ref{genprob}) in $B_R$.
		
	\end{enumerate}

	In the case of $\mathbb{R}^2$ uniqueness is meant up to scaling.\\
\end{theorem}

We recall that, by using the moving plane method, it is proved in \cite{LS} that every positive solution of (\ref{genprob}) in the ball $B_R$ satisfying $u=0$ on $\partial B_R$ is radial. Thus, by Theorem \ref{criticalexp}, we have that such a solution in $B_R$ exists if and only if $ p < p_-^*$.
\par Then, we could study the asymptotic behavior of $u_p$, as $p \nearrow p_-^*$, to understand the limit profile of these solutions.
\par As for the higher dimensional case (\cite{BGLP}) we get:

\begin{theorem} \label{convergenceminus}
	
	Let $N=2$, $p_-^*$ as in Theorem \ref{criticalexp} and $\varepsilon >0$. Then, for the solution $u_{p_\varepsilon}$ of (\ref{genprob}) (for $\mathcal{M}_{\lambda,\Lambda}^-$) in the ball $B_R$,  with $p_\varepsilon= p_-^*-\varepsilon$, the following statements hold:
	
	\begin{enumerate}[i)]
		\item $ M_{\varepsilon} = u_{p_\varepsilon}(0) \rightarrow \infty$ 
		\item the rescaled function
		
		\begin{align*}
		u_\varepsilon = \frac{1}{u_{p_\varepsilon}(0)}u_{p_\varepsilon}\left(\frac{x}{u_{p_\varepsilon}(0)^{\frac{p_\varepsilon -1}{2}}}\right)
		\end{align*}

	converges up to a subsequence, to a limit function $U$ in $C^2_{loc}(\mathbb{R}^2)$ where $U$ is the unique positive solution of:
	
	\begin{align*}
	-\mathcal{M}_{\lambda,\Lambda}^- (D^2 U)(x) = U^{p_{-}^*} \qquad \text{ in } \mathbb{R}^2
	\end{align*} 
	satisfying $U(0)=1$
	\item $u_{p_\varepsilon} \rightarrow 0$ in $\mathcal{C}_{loc}^2 (\bar{B}_R \setminus \{0\})$.
\end{enumerate}
\end{theorem}
The proof of the previous theorem is similar to that of \cite{BGLP} for the analogous results in higher dimension, though the statement iii) could be obtained more easily analyzing the dynamical system introduced in \cite{MNP}. Also, the results about the energy invariance of \cite{BGLP} can be easily extended to the two-dimensional case.
\par Finally, a classification of solutions of (\ref{genprob}) singular at the origin similar to the one of Theorem 1.8 of \cite{MNP} follows in the same way, with obvious changes.

\par We conclude with some remarks about higher dimensions which are related to our results for $N=2$. 
\par First when the dimension $N$ is greater than two, we may consider the case where $\tilde{N}_+$ as defined in (\ref{Dimension}) is smaller than two. Then the results of \cite{CL1} and \cite{QS} still imply that there is, for every $ p \in (1,\infty)$, a unique radial positive solution $u_p$ of (\ref{prob}) in the ball $B_R$. 
\par In particular, the approach we use to treat the two dimensional problem may be immediately applied in this setting producing results analogous to those of Theorem \ref{conv}.
Thus we get solutions of the Liouville equation (\ref{liouv}) in some higher dimensions. Note that even for the semilinear case, where the Pucci operator $\mathcal{M}_{\lambda,\Lambda}^+$ is replaced by the Laplacian the corresponding equation (\ref{liouv}) has different features according to the dimension (see \cite{DF}, \cite{JL} and references therein). It would be interesting to investigate problem (\ref{liouv}) in the fully nonlinear setting in all dimensions and also for the operator $\mathcal{M}_{\lambda,\Lambda}^-$. 
\par Concerning the operator $\mathcal{M}_{\lambda,\Lambda}^-$, the approach we have used in Section 3 to estimate $p_{-}^*$ as in (\ref{expbound}) can also be used in higher dimensions to have an estimate of the critical exponent from above better than the one given in (\ref{ce}) and to determine the existence of oscillating solutions. This will be done in a forthcoming paper.

\par The paper is organized as follows. In Section 2 we study the equation (\ref{genprob}) for the operator $\mathcal{M}_{\lambda,\Lambda}^+$ proving Theorem \ref{conv}. Section 3 is devoted to the Pucci Operator $\mathcal{M}_{\lambda,\Lambda}^-$. After recalling several preliminaries about an associated dynamical system we prove a result on the nonexistence of periodic orbits for such a system. This allows to prove Theorem \ref{criticalexp}. We end Section 3 with the proof of Theorem \ref{convergenceminus}. 

\section{The Problem for $\mathcal{M}^+$}
\subsection{Some Preliminary Results}
Here we present some needed classical results.

\begin{theorem}[Pucci's lemma \cite{CP}]\label{PL} 
		Let $\Omega \subset \mathbb{R}^N$ be open and let $u: \Omega  \rightarrow \mathbb{R}$ be a $\mathcal{C}^2(\Omega)$ function. Then for every $x \in \Omega$ there is a ($\lambda$, $\Lambda$) elliptic matrix $A(x)$ depending on $u$ such that:
	
	\begin{align}
	\mathcal{M}_{\lambda,\Lambda}^+ (D^2 u)(x) = Tr(A(x) D^2 u(x))
	\end{align}
	
	\noindent and $A(x)$ is a measurable function with respect to $x$.
\end{theorem}

\begin{theorem}[\cite{GT}]\label{harnack}
	Let $L$ be an uniformly elliptic linear operator with measurable bounded coefficients, and $\Omega,u$ as above. If $u \geq 0 $ and $Lu =0$ in $\Omega$, then for every $\Omega' \subset \subset \Omega$ we have
	
	\begin{align}
		\sup_{\Omega'} u \leq C \inf_{\Omega} u
	\end{align}
\end{theorem}

\begin{theorem}[\cite{GT}] \label{estimate}
	Let $L,\Omega$ and $u$ be as in Theorem \ref{harnack}. If $Lu = f $ in $\Omega$, for some $f \in \mathcal{L}_{loc}^n(\Omega)$, then for every $\Omega' \subset \subset \Omega$ we have
	
	\begin{align}
	||u||_{\mathcal{C}^\alpha(\Omega')} u \leq C \left(\sup_{\Omega} u + ||f||_{\mathcal{L}_{loc}^n(\Omega)}\right) 
	\end{align}
\end{theorem} 
\vspace{0.5cm}
\par For $ p > 1 $, we consider the Dirichlet problem (\ref{prob}) and recall known results for it.

\begin{theorem} \label{prop}
	For every $p>1$ the problem (\ref{prob}) admits a unique solution $u_p$ which is radial, i.e, with an abuse of notation, $u_p(x) = u_p(|x|)$ for $\vert x \vert = r \in [0,R]$. Furthermore $u_p$ satisfies:
	
	\begin{enumerate}[i)]
		\item $u_p(0) = \max u_p$
		\item $u_p$ is strictly decreasing
		\item $u_p$ changes concavity only once at a point $y_p$ and is concave around the origin.
	\end{enumerate}
\end{theorem}

\begin{proof}
	The existence of a positive solution of (\ref{prob}) for every $p>1$ derives from the nonexistence of solutions of the analogous equation in $\mathbb{R}^2$ (see \cite{CL1}). Indeed if entire solutions do not exist, then apriori estimates hold which allow to prove the existence of a solution of (\ref{prob}) as in \cite{QS}.
	\par The  radial symmetry of $u_p$ and i)-ii) have been proved in \cite{LS} by the moving plane method. The uniqueness follows by the invariance by scaling of the equation and the uniqueness of the initial value problem for the corresponding ODE.
	\par Finally, iii) can be proved exactly as in \cite[Lemma 3.1]{FQ1} using the Emden-Fowler analysis.    
\end{proof}

\begin{Remark}
	The uniqueness of the radius $y_p$ where $u''(r)=0$ can be obtained also as in \cite{MNP} analyzing the flow induced by an associated dynamical system. This method also works in dimension two, and we will use it to study the problem for $\mathcal{M}_{\lambda,\Lambda}^-.$
\end{Remark}

\subsection{Asymptotic behavior of rescaled solutions}
	In this section we prove Theorem \ref{conv}. 
	
	We are interested in the asymptotic behavior of the solution $u_p$ when the exponent $p$ goes to $+\infty$.

%
%
%

Recalling that  the parameter $\varepsilon$ is defined by

\begin{equation*}
	\varepsilon_p^{-2} = p \cdot u_p(0)^{p-1}
\end{equation*}
we prove the following preliminary result.
\begin{lemma}
	It holds :
	$$\lim\limits_{p\rightarrow +\infty}{\varepsilon_p^{-2}}= +\infty$$
\end{lemma}
\begin{proof}
	The proof is based on the fact that $u_p(0)^{p-1}$ does not converge to 0. Assume otherwise, then for $p$ sufficiently big, we have:
	
	\begin{align*}
 		u_p(0)^{p-1} < \lambda_1 (\mathcal{M}_{\lambda,\Lambda}^+, B_R)		
	\end{align*}
	where $\lambda_1(\mathcal{M}_{\lambda,\Lambda}^+, B_R)$ is the first eigenvalue of the Pucci operator on $B_R$ with homogeneous Dirichlet boundary conditions.
	\par In particular, this implies that

\begin{equation*}
\begin{cases} 
-\mathcal{M}_{\lambda,\Lambda}^+ (D^2 u_p)(x) = u_p(x)^p < \lambda_1 (\mathcal{M}_{\lambda,\Lambda}^+, B_R)\cdot u		 \qquad \text{ in } B_R\\
u_p > 0 \qquad \text{ in } B_R  \\
u_p = 0 \qquad \text{ on } \partial B_R  
\end{cases}
\end{equation*}
	
This is a contradiction with the definition of the first eigenvalue. Therefore we conclude that $\lim\limits_{p \rightarrow \infty}p \cdot u_p(0)^{p-1} = \infty$.
	
\end{proof}

\begin{Proposition} \label{convprop}
	
The rescaled function $z_p$ defined in (\ref{rescale}) converges in $C^2_{loc}(\mathbb{R}^2)$, up to a subsequence, to a solution $z$  of the Liouville equation (\ref{liouv}).
\end{Proposition}

\begin{proof}
	A simple computation gives for $x \in B_{R/\varepsilon_p}:$
	\begin{align*}
		 D^2z_p(x) = \frac{p}{u_p(0)} \cdot \varepsilon_p^2 \cdot D^2u_p(\varepsilon_p \cdot x)
	\end{align*}
	
	Therefore, 
	\begin{align*}
	\mathcal{M}_{\lambda,\Lambda}^+(D^2z_p)(x) = \frac{p}{u_p(0)} \cdot \varepsilon_p^2 \cdot \mathcal{M}_{\lambda,\Lambda}^+(D^2u_p)(\varepsilon_p \cdot x)
	\end{align*}
	
	Since $u_p$ is a solution to Problem (\ref{prob}), we get
	
	\begin{align*}
	\mathcal{M}_{\lambda,\Lambda}^+(D^2z_p)(x) = \frac{-p}{u_p(0)} \cdot \varepsilon_p^2 u_p^p(\varepsilon_p \cdot x)=-\left( 1 +\frac{z_p}{p} \right) ^p 
	\end{align*}
	
	It is clear from the above expression that we are done once we prove that $z_p$ converges up to a subsequence in $C^2_{loc}(\mathbb{R}^2)$. \\

	 It follows from Theorem \ref{PL} that for every $p$, there is a $(\lambda,\Lambda)$ elliptic matrix $A_p(x)$ such that
	
	\begin{align}
		-\mathcal{M}_{\lambda,\Lambda}^+(D^2z_p)(x) = -Tr(A_p(x)D^2z_p(x))=\left( 1 +\frac{z_p(x)}{p} \right) ^p 
	\end{align} 
	
	Let $B_{R_1}$ be the ball of radius $R_1>0$ centered at the origin and consider $w$ the solution of the problem:

\begin{equation}
\begin{cases}
-Tr(A_p(x)D^2w_p(x))=\left( 1 +\frac{z_p(x)}{p} \right) ^p  \qquad \text{ in } B_{R_1}\\
w_p = 0 \qquad \text{ on } \partial B_{R_1}  
\end{cases}
\end{equation}	
	 
It follows from the definition of $z_p$ that $(1+\frac{z_p}{p})$ $\in (0,1]$. Therefore it follows from the Alexandrov-Bakelman-Pucci (\cite{GT}) estimate and maximum principle that $0 \leq w_p \leq C(\lambda,R_1)$.

For $x \in B_{R_1}$ define the auxiliary function $\varPsi_p (x) = z_p(x) - w_p(x)$ which solves the problem

\begin{equation}	
\begin{cases}
-Tr(A_p(x)D^2\varPsi_p(x))= 0  \qquad \text{ in } B_{R_1}\\
\varPsi_p = z_p \qquad \text{ on } \partial B_{R_1}.  
\end{cases}
\end{equation}	

It follows from  Theorem \ref{harnack}, applied to $-\varPsi_p$, that there is a constant $C_1(\Lambda / \lambda,R_1)$ such that

\begin{align*}
	C_1\cdot  \varPsi_p(0) \leq C_1\cdot  \sup_{{ B_{R_1}}} \varPsi_p \leq \inf_{{ B_{R_1/2}}} \varPsi_p.
\end{align*}

Note that $\varPsi_p(0) = z_p(0) - w_p(0) \geq C$. Therefore we conclude that

\begin{align*}
	-\inf_{{ B_{R_1/2}}} \varPsi_p =\sup_{B_{R_1/2}} |\varPsi_p| \leq C(\Lambda,\lambda,R_1).
\end{align*}
	
	Since $w_p$ is also uniformly bounded we obtain:

\begin{align*}
\sup_{B_{R_1/2}} |z_p| \leq C(\Lambda,\lambda,R_1).
\end{align*}
	
	We proceed by using the $\mathcal{C}^\alpha$ estimates (Theorem  \ref{estimate}) to state that there are $\alpha(\lambda,\Lambda)$ and $ C(\Lambda,\lambda,R_1)$ such that 

\begin{align*}
	\sup_{B_{R_1/4}} [z_p]_{{\mathcal{C}^\alpha}} \leq C.
\end{align*}	
	
	Note that 
	
\begin{align*}
\left[1 + \frac{z_p}{p}\right]_{{\mathcal{C}^\alpha}} \leq \left[z_p\right]_{{\mathcal{C}^\alpha}}.
\end{align*}	

Therefore it follows from the $\mathcal{C}^{2,\alpha}$ estimates for the Pucci operator that 

\begin{align*}
\sup_{B_{R_1/8}} |z_p|_{{\mathcal{C}^{2,\alpha}}} \leq C(\Lambda,\lambda,R_1).
\end{align*}	

Thus from the Arzela-Ascoli theorem there is a subsequence $z_{p_k}$ which, for every $\alpha' < \alpha $ locally converges to a limit function z in $\mathcal{C}^{2,\alpha'}$ in $B_{R_2}$, for a sufficiently small $R_2$. Since $R_1$ is arbitrary, the convergence holds locally over the whole plane.
	
\end{proof}

\subsection{The fully nonlinear Liouville equation}

In the previous section, we have shown that, up to a subsequence, the radial functions $z_p$ converge to a function $z$ which solves (\ref{liouv}) which is a fully nonlinear version of the Liouville equation. Now we proceed in describing such a function in particular proving that $z$ changes concavity only once.

\begin{proof}[\textbf{Proof of Theorem \ref{description}}]

The convergence of the rescaled function $z_p$ is just the statement of Proposition \ref{convprop}. Hence the existence of a radial solution $z$ of (\ref{liouv}) which is negative and decreasing is deduced by that. Now we show that the limit function $z$ changes concavity only once. Thus we consider the only radius $y_p$ where $u_p$ changes concavity ( iii) of Theorem \ref{prop} ). 

We consider the three possible cases:

\begin{enumerate}
	\item $\lim\limits_{p\rightarrow \infty} \frac{y_p}{\varepsilon_p} = \infty$.
	\item  $\lim\limits_{p\rightarrow \infty} \frac{y_p}{\varepsilon_p} = 0 $.
	\item  $\lim\limits_{p\rightarrow \infty} \frac{y_p}{\varepsilon_p} = R_0, \qquad 0 < R_0 < \infty$.
\end{enumerate}

Case 1: Since $\lim\limits_{p\rightarrow \infty} \frac{y_p}{\varepsilon_p} = \infty$, the limit equation being satisfied by $z$ is

\begin{equation} \label{case1}
\begin{cases}
	-\lambda \Delta z &= e^z \qquad \text{ in } \mathbb{R}^2 \\
	z(0)&=0,
\end{cases}
\end{equation}

\noindent because $z$ is concave and decreasing on the whole space. From the classification of solution of the Liouville equation by Chen-Li \cite{CL}, we know that the solution of (\ref{case1}) is:

\begin{align}
	z(x) = \log\left( \frac{1}{(1+\frac{1}{8\lambda}|x|^2)^2} \right).
\end{align}

We conclude that Case 1 is not possible since such a solution is not concave in the whole space.\\

Case 2: Since $\lim\limits_{p\rightarrow \infty} \frac{y_p}{\varepsilon_p} = 0$, the limit ODE being satisfied by $z(x) = z(|x|) $ is

\begin{equation}
\begin{cases}
-\Lambda z''(r) - \lambda \frac{z'(r)}{r} = e^{z(r)} \qquad \text{ in } (0,\infty) \\
z(0)=0,
\end{cases}
\end{equation}
\noindent and $z$ is a nonpositive, decreasing and convex function.

First observe that we may rewrite the above equations as

\begin{align*}
	-(r^{\frac{\lambda}{\Lambda}}\varphi'(r))'=\frac{r^{\lambda/\Lambda}}{\Lambda}e^{\varphi(r)} \qquad r \in (0,\infty).
\end{align*}

Integrating between $0$ and $r$ since $z$ is negative so that $e^z \leq 1$ we obtain:

\begin{align*}
r^{\lambda/\Lambda} \left( z'(r) \right)= \int_{0}^{r} \frac{-s^{\lambda / \Lambda}}{\Lambda}e^{z(s)}ds\geq \frac{-1}{\Lambda}r^{1 +\frac{\lambda}{\Lambda}} \frac{\Lambda}{\lambda+ \Lambda}.	
\end{align*}
Dividing by $r^{\lambda/\Lambda}$ we obtain

\begin{align*}
	0 \geq z'(r) \geq \frac{-1}{\lambda+\Lambda}\cdot r,
\end{align*}
and taking the limit as $r$ goes to zero we get $z'(0)=0$. In particular this implies that $z$ is positive around the origin since $z(0)=0,z'(0)=0$ and $z$ is convex. This contradicts the fact that $z$ is negative.\\

Case 3: Since $\lim\limits_{p\rightarrow \infty} \frac{y_p}{\varepsilon_p} = R_0$, the limit equations being satisfied by $z(x) = z(|x|)$ are

\begin{equation}\label{case31}
\begin{cases}
-\lambda z''(r) - \lambda \frac{z'(r)}{r} = e^{z(r)} \qquad \text{ in } (0,R_0) \\
-\Lambda z''(r) - \lambda \frac{z'(r)}{r} = e^{z(r)} \qquad \text{ in } (R_0,\infty) \\
z(0)=0 \\
z''(R_0)=0,
\end{cases}
\end{equation}

and $z$ is a nonpositive, decreasing function, concave in $(0,R_0)$ and convex in $(R_0,\infty)$.

As in the previous case, we may repeat the same procedure by exchanging $\Lambda$ by $\lambda$ and obtain that $z'(0)=0$.

Therefore $z$ satisfies the following initial boundary value problem

\begin{equation}\label{case32}
\begin{cases}
-\lambda z''(r) - \lambda \frac{z'(r)}{r} = e^{z(r)} \qquad \text{ in } (0,R_0) \\
z(0)=0 \\
z'(0)=0
\end{cases}
\end{equation}

Since such a problem admits only one solution it must be given by

\begin{align*}
	z(r) = \log\left( \frac{1}{(1+\frac{1}{8\lambda}r^2)^2} \right)
\end{align*}

Note that the above function changes concavity only once in $\mathbb{R}^+$ at $2\cdot\sqrt{2\cdot\lambda}$. Therefore $R_0=2\cdot\sqrt{2\cdot\lambda}$. In particular, due to the fact that $z$ is $\mathcal{C}_{loc}^2(\mathbb{R}^2)$, we obtain the following:

\begin{itemize}
	\item $z(R_0)=-2\cdot \ln2$
	\item $z'(R_0)= -\frac{1}{\sqrt{2\lambda}}$
\end{itemize}

Then by (\ref{case31}) $z$ satisfies this other initial value problem

\begin{equation} \label{case33}
\begin{cases}
\Lambda z''(r) - \lambda \frac{z'(r)}{r} = e^{z(r)} \qquad \text{ in } (0,R_0) \\
z(R_0)=-2\cdot \ln2 \\
z'(R_0)= -\frac{1}{\sqrt{2\lambda}}
\end{cases}
\end{equation}
There is only one solution to problem (\ref{case33}) and therefore we obtain that $z$ is given by gluing the solutions of the two initial boundary value problems (\ref{case32}) and (\ref{case33}).
\end{proof}

\section{The Problem for $\mathcal{M}^-$}

We consider the problem\\

\begin{equation}
\begin{cases} \label{probmenos}
-\mathcal{M}_{\lambda,\Lambda}^- (D^2 u)(x) = u(x)^p  \qquad \text{ in } \Omega\\
u > 0 \qquad \text{ in } \Omega  \\
\end{cases}
\end{equation} \\

\noindent where $\Omega$ is either $\mathbb{R}^2$ or the ball $B_R$, centered at the origin, with radius $R>0$. In the last case we assume:

\begin{equation}
	u = 0 \qquad \text{ on } \partial \Omega  
\end{equation}



The first step to study the above problem is understanding whether a critical exponent for the existence of solutions of (\ref{probmenos}) can be defined or not.

To this aim we are going to use the approach of \cite{MNP} which involves an auxiliary dynamical system. Thus we start by introducing it together with some preliminary results from \cite{MNP}.

\subsection{Preliminaries on a Dynamical System}

Since we are interested in radial solutions to (\ref{probmenos}) we write $u=u(r) = u(|x|)$ as an expression of $u$ in radial coordinates. The eigenvalues of $D^2 u$ are the simple eigenvalue $u''(r)$ and $\frac{u'(r)}{r}$ which has multiplicity $(N-1)$ (see \cite{FQ}).
\par Thus $u$ satisfies a corresponding ODE from which it is easy to deduce that $u$ is decreasing as long as it is positive, concave in an interval $(0, \tau_0)$ and changes concavity at least once (see \cite{FQ} or \cite{MNP})

Hence (\ref{probmenos}) can be written as:

\begin{equation}\label{ODE}
	u''= M(-\Lambda r^{-1} u' -u^p),
\end{equation}

\noindent where

\begin{equation}
	M(s)= \begin{cases}
		\frac{s}{\Lambda} \qquad \text{ if } s \leq 0 \\
		\frac{s}{\lambda} \qquad \text{ if } s > 0 \\
	\end{cases},
\end{equation}

\par Next we introduce the following auxiliary functions:
\begin{align} \label{variables}
	X(t) = -\frac{r u'(r)}{u(r)}, \qquad Z(t) = - \frac{r u(r)^p}{u'(r)}  \qquad \text{for }t=\ln r
\end{align} 
whenever $ u(r) \neq 0 , u'(r) \neq 0$ \\

Since $\,u > 0\,$ and $\,u'<0\,$, we have that the above quantities belong to the first quadrant of the $(X,Z)$ plane.

In the new variables, the equation (\ref{ODE}) becomes the following autonomous dynamical system

\begin{align}\label{system}
	\left(\dot{X},\dot{Z}\right) = F(X,Z)= \left(f(X,Z),g(X,Z)\right)	
\end{align}

\noindent where the dot $\cdot$ stands for derivation with respect to t and $f,g$ are given by:

\begin{align*}
	f(X,Z) =
	\left\{
	\begin{array}{ll}
		X(X+\frac{Z}{\Lambda})  & \mbox{if } (X,Z) \in R^+ \\
		X(X-(\tilde{N}_{-} -2) +\frac{Z}{\lambda}) & \mbox{if } (X,Z) \in R^-
	\end{array}
	\right.
\end{align*}
\\
\begin{align*}
	\hspace{-1cm}g(X,Z) =
	\left\{
	\begin{array}{ll}
		Z(2-pX-\frac{Z}{\Lambda})  & \mbox{if } (X,Z) \in R^+ \\
		Z(\tilde{N}_{-} -pX-\frac{Z}{\lambda}) & \mbox{if } (X,Z) \in R^-
	\end{array}
	\right.
\end{align*}

\noindent where the regions $R^+$ and $R^-$ are:

\begin{align}
	R^+ &= \{(X,Z) \,|\, Z > \Lambda \} \\
	R^- &= \{(X,Z)  \,|\, 0 < Z < \Lambda \}
\end{align}

Note that whenever $(X(t),Z(t))$ belongs to the line
\begin{equation}
	\ell = \{(X,Z  \,|\, Z=\Lambda)\}
\end{equation}

\noindent then the corresponding solution $u$ of (\ref{ODE}) satisfies $u''=0$. Hence $R^+$ and $R^-$ represent, in terms of the new variables $(X,Z)$, the regions of concavity and convexity of $u$, respectively.
\par Other important sets  which are relevant to study the dynamics induced by (\ref{system}) are: 

\begin{equation}\label{l1}
	\ell_1= \{(X,Z)  \,|\, Z= \lambda(\tilde{N}_{-} -2-X)\} 
\end{equation}

\noindent which is the set where $\dot{X}=0$, and
\begin{equation}
 	\ell_2 = \{(X,Z) \in R^+ \,|\, Z= \Lambda({2} -pX)\} \cup \{(X,Z) \in R^- \,|\, Z= \lambda(\tilde{N}_{-} -pX)\} 
\end{equation}

\noindent which is the set where $\dot{Z}=0$.

Then we need to consider the stationary points for (\ref{system}) and their classification.
\par We recall that for $p \leq \frac{\tilde{N}_-}{\tilde{N}_--2}$ is already known that there exists a positive solution of (\ref{probmenos}) with Dirichlet boundary conditions in the ball $B_R$ (\cite{QS}).
\par Therefore, to determine the critical exponent for $\mathcal{M}_{\lambda,\Lambda}^-$ only the values of $p$ larger than $\frac{\tilde{N}_-}{\tilde{N}_- -2}$ are important, hence, we just state all is needed for $p>\frac{\tilde{N}_-}{\tilde{N}_- -2}$.

\begin{Proposition}
The ODE system (\ref{system}) admits the following stationary points:

\begin{enumerate}[i)]
	\item $N_0=(0,2\Lambda)$ which is a saddle point.
	\item $A_0 = (\tilde{N}_{-} -2,0)$ which is a saddle point.
	\item $M_0=\left(\frac{2}{p-1},\lambda(\tilde{N}_{-} -2-\frac{2}{p-1})\right)$  which is a source for $ \frac{\tilde{N}_{-} }{\tilde{N}_{-} -2} <p < \frac{\tilde{N}_{-} +2}{\tilde{N}_{-} -2}$, a center for $p = \frac{\tilde{N}_{-} +2}{\tilde{N}_{-} -2}$ , a sink for $p > \frac{\tilde{N}_{-} +2}{\tilde{N}_{-} -2}$
	\item $O=(0,0)$ which is a saddle point.  
\end{enumerate}

\end{Proposition}

The stationary points and the direction of the vector field $F$ on the relevant sets for (\ref{system}) are displayed in Figure 1.

\subsection{Periodic Orbits}

For the sequel it is important to see for which range of the exponent $p$ there are no periodic orbits of (\ref{system}).

\begin{theorem}\label{porbits}
	
	Let $0<\lambda \leq \Lambda$. Then for the system (\ref{system}) it holds:
	\begin{enumerate}[i)]
		\item no periodic orbits exist if $\,1 < p < \frac{\tilde{N}_{-} +2}{\tilde{N}_{-} -2}.$ 
		\item there exists $\,p_o > \max \{3,\frac{\tilde{N}_{-} +2}{\tilde{N}_{-} -2}\}\,$ such that for $\,p > p_o $ there is no periodic orbit.
	\end{enumerate}	
\end{theorem}

\begin{proof}
	The statement i) was already proved in Proposition 2.10 of \cite{MNP}. We include the proof here for the reader's convenience.
	\par Note that a periodic orbit must necessarily contain a stationary point in its interior. In our case, the only one which can be in it is the point $M_0$. Moreover by the direction of the vector field $F$ (see (\ref{system})) on the lines $L_1 = \{(X,Z) \, | \, Z=2\Lambda\}$ and $L_2 = \{(X,Z) \, | \, X=\tilde{N}_- - 2\}$ we deduce that no periodic orbit can intersect them. Hence, a periodic orbit must be contained in $[0,\tilde{N}_- - 2] \times [0, 2\Lambda]$. 
	\par Consider $\varphi(X,Z)= X^\alpha Z^\beta$, where $\beta = \frac{3-p}{p-1}$. Set $\Phi(X,Z)= \partial_X(\varphi f)+\partial_Z(\varphi g)$, with $f$ and $g$ as in (\ref{system}), i.e 
	
	\begin{align*}
		\Phi(X,Z) =
		\left\{
		\begin{array}{ll}
		X^\alpha Z^\beta \frac{4}{p-1}  & \mbox{if } (X,Z) \in R^+ \\
		- X^\alpha Z^\beta \frac{p(\tilde{N}_{-} -2)-(\tilde{N}_{-} +2)}{p-1} & \mbox{if } (X,Z) \in R^-
		\end{array}
		\right.
	\end{align*}

	Suppose there is a periodic orbit which encloses a region D. It follows from the Green's area formula that:
	
	\begin{align*}
	0 =\int_{\partial D} \varphi \{  f dX - g dZ\}= \int_{D} \Phi dX dZ = \int_{R^+ \cap D} \Phi dX dZ + \int_{R^- \cap D} \Phi dX dZ,
	\end{align*}
	where the first equality holds since $dX = f dt$ and $dZ = g dt$. \\
	From this it follows that for $ p \neq \frac{\tilde{N}_- +2}{\tilde{N}_- -2} $, any periodic orbit must intersect both regions $R^+,R^-$.\\
	It is clear that, for $p$ satisfying $1 < p < \frac{\tilde{N}_{-} +2}{\tilde{N}_{-} -2},$ $\Phi$ is positive and the above identity is a contradiction. Hence i) holds.
	
	Now assume $p > 3$ (which implies that $\beta < 0$). In the region $R^+$, we have $\Lambda < Z < 2\Lambda$, which implies $Z^\beta < \Lambda^\beta$. Also, in $R^-$, $0 < Z < \Lambda$, hence $-Z^\beta < -\Lambda^\beta$. Hence the following bound holds:
	\begin{multline}\label{intineq}
	0 = \int_{R^+ \cap D} \Phi dX dZ + \int_{R^- \cap D} \Phi dX dZ \leq \\
	\frac{4\Lambda^\beta}{p-1} \int_{R^+ \cap D} X^\alpha dX dZ - \frac{p(\tilde{N}_{-} -2)-(\tilde{N}_{-} +2)}{p-1}\Lambda^\beta \int_{R^- \cap D} X^\alpha dX dZ
	\end{multline}
	
	In order to conclude the proof it is sufficient to check that the ratio of the areas $\vert \frac{R^+ \cap D \vert}{\vert R^- \cap D\vert }$ $\nrightarrow$ $\infty $ as $p$ goes to infinity. Indeed if this holds, taking the limit as $p$ goes to infinity we get a contradiction and the result follows.
	
	\begin{figure}[H]
		\centering
		\includegraphics[width=1.0\linewidth]{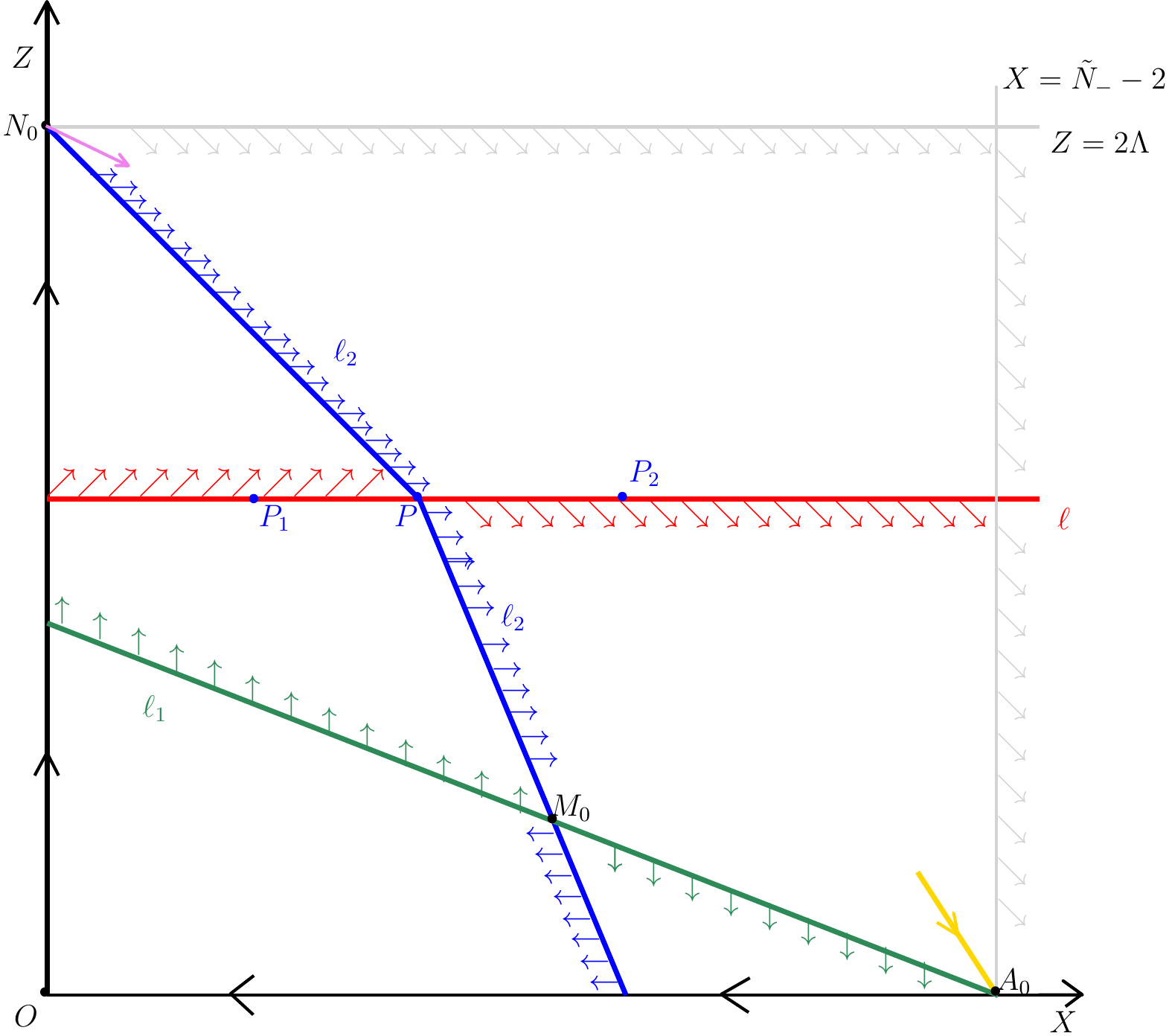}
		\caption{Indicative of the flow for $p > \frac{\tilde{N}_{-} }{\tilde{N}_{-} -2}$}
		\label{fig:capture}
	\end{figure}
	
	Let $P_1=(x_1,\Lambda),P_2=(x_2,\Lambda)$ be  two points where the orbit intersects the concavity line $\ell$, note that since $P=(\frac{1}{p},\Lambda)$ is the only point where $\dot{Z}$ is zero, this implies the uniqueness of $P_1, P_2$. Since $\dot{X} > 0 $ in $R^+$, it follows that $R^+ \cap D$ is contained in the rectangle $Q_1 = [x_1,x_2] \times [\Lambda, 2\Lambda]$. Also since $\dot{X} > 0 $ above the line $\ell_1$ defined in (\ref{l1}) the rectangle $Q_2 = [x_1,x_2] \times [\lambda(\tilde{N}_{-} -2),\Lambda] $ is contained in $R^- \cap D$.
	
	Therefore
	
	\begin{align*}
	\frac{\vert R^+ \cap D \vert}{\vert R^- \cap D \vert} \leq \frac {(x_2-x_1) \Lambda}{(x_2-x_1) (\Lambda - \lambda(\tilde{N}_{-} -2))} = \frac{\Lambda}{\Lambda - \lambda(\tilde{N}_{-} -2)}. 
	\end{align*}
	
	Therefore the right hand side cannot converge to $+\infty$ as $p \rightarrow \infty $, so the existence of $p_o$ satisfying ii) is guaranteed.
		
\end{proof}


 \begin{equation}\label{pmenos}  \text{Define } \tilde{p}_{-} = \inf \left\{ p \in (\frac{\tilde{N}_{-} +2}{\tilde{N}_{-} -2},\infty) \;|\; \nexists \text{ periodic orbits } \forall p' > p \right\}
 \end{equation}

\begin{theorem}[Bound for $\tilde{p}_{-}$]\label{boundpmenos}
	It holds: $$\frac{\tilde{N}_- +2}{\tilde{N}_- -2}\leq\tilde{p}_{-} \leq  \frac{\tilde{N}_{-} +2}{\tilde{N}_{-} -2} + \frac{4}{\tilde{N}_{-} -2+\frac{\lambda}{\Lambda}(\tilde{N}_{-} -2)^2}$$.
\end{theorem}

\begin{proof}
	If $ p > 3 $ and a periodic orbit exists, then, going back to (\ref{intineq}), it holds: 
	\begin{align*}
			0  \leq 
		\frac{4\Lambda^\beta}{p-1} \int_{R^+ \cap D} X^\alpha dX dZ - \frac{p(\tilde{N}_{-} -2)-(\tilde{N}_{-} +2)}{p-1}\Lambda^\beta \int_{R^- \cap D} X^\alpha dX dZ
	\end{align*}
	
	Considering the rectangles $Q_1$ and $Q_2$ defined in the proof of Theorem \ref{porbits}, we have:  
	
	\begin{multline*}
		\int_{R^+ \cap D} X^\alpha dX dZ \leq \int_{Q_1} X^\alpha dX dZ \leq  \frac{\Lambda}{\Lambda - \lambda(\tilde{N}_{-} -2)}\int_{Q_2} X^\alpha dX dZ \leq \\ \frac{\Lambda}{\Lambda - \lambda(\tilde{N}_{-} -2)}\int_{R^- \cap D} X^\alpha dX dZ
	\end{multline*}
	
 Therefore 
 \begin{align*}
 	0 \leq \left( \frac{\Lambda}{\Lambda - \lambda(\tilde{N}_{-} -2)}\cdot \frac{4\Lambda^\beta}{p-1}   - \frac{p(\tilde{N}_{-} -2)-(\tilde{N}_{-} +2)}{p-1}\Lambda^\beta \right) \int_{R^- \cap D} X^\alpha dX dZ 
 \end{align*}
 
 Observing that the term in brackets is negative whenever $p > \frac{\tilde{N}_{-} +2}{\tilde{N}_{-} -2} + \frac{4}{\tilde{N}_{-} -2-\frac{\lambda}{\Lambda}(\tilde{N}_{-} -2)^2}$, which is greater than 3, we get (\ref{boundpmenos}). 
 
\end{proof}

\subsection{Critical exponent}

We recall that for an orbit $\tau$ of the dynamical system (\ref{system}) the set of limit points of $\tau(t)$, as $t \rightarrow -\infty$, is usually called $\alpha$-limit and denoted by $\alpha(\tau)$. Analogously it is defined the $\omega$-limit $\omega(\tau)$ at $+\infty$.
With the same proofs as in \cite{MNP} we have

\begin{lemma}\label{box}
	For every $p > 1$, any regular solution of (\ref{ODE}) satisfying  the initial conditions: $u_p(0)= \gamma > 0 \, , \, u_p'(0 )= 0$, corresponds to the unique trajectory $\Gamma_p$ of (\ref{system}) whose $\alpha$-limit is the stationary point $N_0$. Moreover:
	\begin{enumerate}[i)]
		\item if $u_p = u_p(r)$ is positive $\forall r \in (0,+\infty)$ then $\Gamma_p$ is bounded and remains in the rectangle $(0,\tilde{N}_- -2) \times (0,2\Lambda)$, for all time.
		\item if there exists $R_p$ such that $u_p(R_p) = 0$ and $u_p >0$ in $(0,R_p)$, then the corresponding trajectory $\Gamma_p=(X_p,Z_p)$ blows up in finite time $T$, in particular:
		
		\begin{equation*}
			\lim\limits_{t\rightarrow T} X_p(t) = +\infty \qquad \lim\limits_{t\rightarrow T} Z_p(t) = 0
		\end{equation*}
	\noindent and there exists $t_1 < T$ such that $X_p(t_1) = \tilde{N}_- -2$.
	\end{enumerate}
\end{lemma} 

\begin{proof}
	See Proposition 2.1, Proposition 3.6 and Proposition 3.9 of \cite{MNP}. 
\end{proof}

We prove now a crucial result about nonexistence of solutions for the problem (\ref{probmenos}) in the ball. It is the keypoint to define and estimate the critical exponent. 

\begin{theorem} \label{soluball}
	For $ p > \tilde{p}_-$, defined in (\ref{pmenos}), there are no solutions of (\ref{probmenos}) in the ball.
\end{theorem}

\begin{proof}
	We recall that $ \tilde{p}_- \geq \frac{\tilde{N}_- +2}{\tilde{N}_-2}.$ Assume that there is a solution of (\ref{probmenos}) in the ball for some $ p > \tilde{p}_-$. Then we would have a trajectory $\Gamma_p = \Gamma_p(t) = (X_p(t),Z_p(t))$ starting from $N_0$ such that $\Gamma_p$ blows up in finite time, after intersecting the line $L_1 = \{(X,Z) \, | \, X = \tilde{N}_- -2\}$ (see Lemma \ref{box}). On the other hand, if we consider the unique trajectory $\tau_p$ whose $\omega$-limit is $A_0$ (which is a saddle point for $p > \frac{\tilde{N}_-}{\tilde{N}_- -2})$ and backtrack it we should be in either one of the following cases:
	\begin{enumerate}[a)]
		\item $\alpha(\tau_p)$ is a stationary point or a periodic orbit around $M_0$.
		\item $\tau_p$ blows up backward in finite time.
	\end{enumerate}

\noindent The case a) is not possible. Indeed $M_0$ is a sink and no periodic orbit exists for $ p > \tilde{p}_-$, by definition (\ref{pmenos}). Moreover $N_0$ cannot be the $\alpha$-limit of $\tau_p$ because $N_0$ is a nondegenerate saddle point, so that, only the trajectory $\Gamma_p$ starts from $N_0$ and $\omega(\Gamma_p)$ is not $A_0$ since $\Gamma_p$ blows up in finite time.\\
Also case b) is not possible because $\tau_p$ is necessarily bounded because it stays in the regions enclosed by the $X$ axis, $Z$ axis, the orbit $\Gamma_p$ and the line $L_1$ from which any orbit can only exit in forward time (by the direction of the vector field on $L_1$), see Figure \ref{fig:barrier}.

	\begin{figure}[H]
	\centering
	\includegraphics[width=0.7\linewidth]{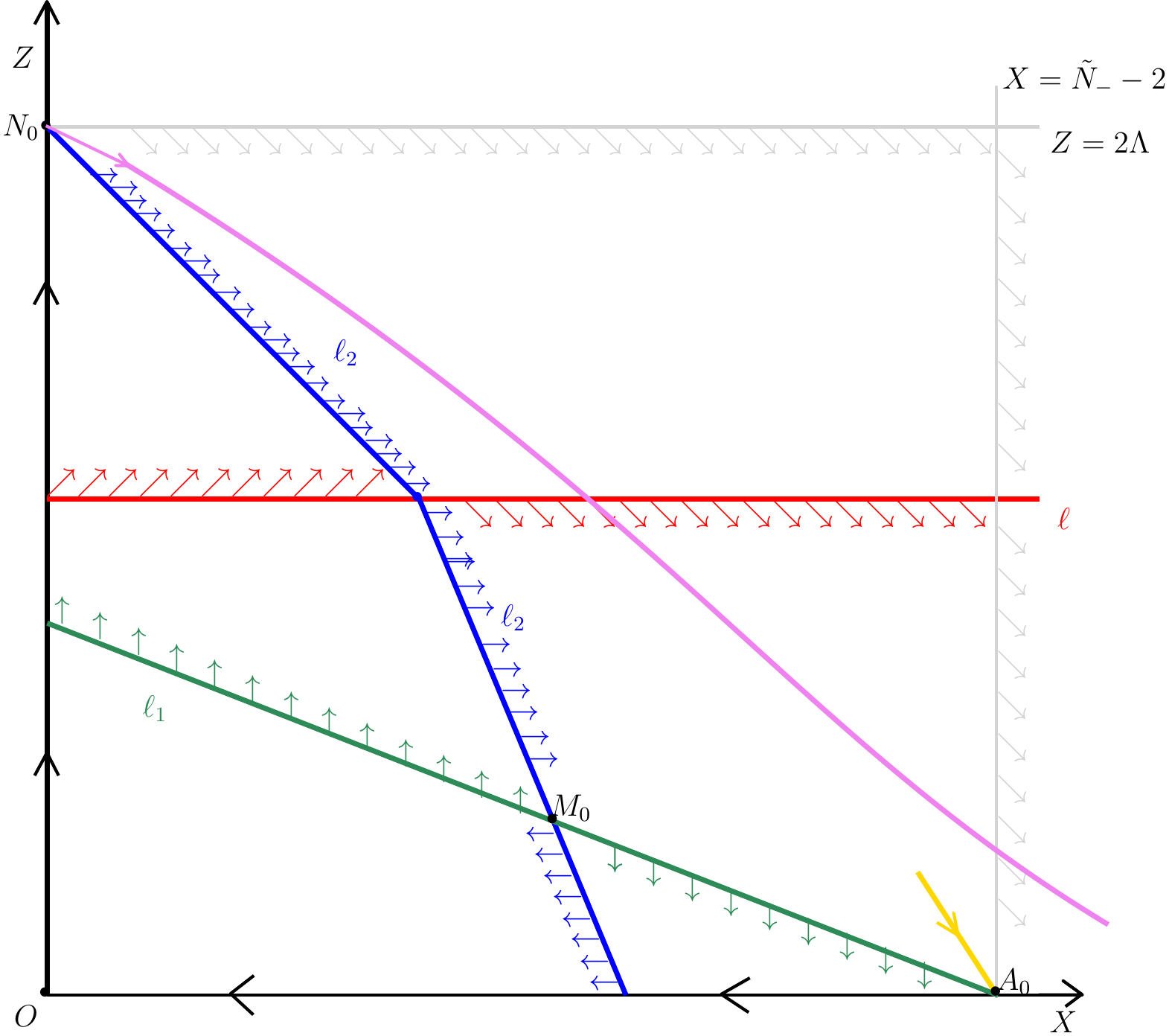}
	\caption{}
	\label{fig:barrier}
\end{figure}
 
 Thus we have a contradiction which implies that $\Gamma_p$ cannot correspond to a solution in the ball.
 
\end{proof}

Then we consider the set:
\begin{equation}
\mathcal{C} = \{ p > 1 \, | \, \text{ a solution of (\ref{probmenos}) in  $B_R$ exists.}\}	
\end{equation}

\noindent and we observe that $\mathcal{C}$ is nonempty since we have already remarked that for $ p \in (1, \frac{\tilde{N}_-}{\tilde{N}_- -2}]$ (\ref{probmenos}) has a solution in $B_R$.
\par So as in \cite{MNP} we define:

\begin{equation}
	p_-^* \,=\, \sup \mathcal{C}\, <\, \infty
\end{equation}

\noindent and call it the critical exponent for $\mathcal{M}_{\lambda,\Lambda}^-$. \\
Before proving Theorem \ref{criticalexp} we recall the following definitions.

\begin{Definition}
	Let $u_p$ be a radial solution of (\ref{probmenos}) in $\mathbb{R}^2$ and $\alpha= \frac{2}{p-1}$. Then $u_p$ is said to be:
	
	\begin{enumerate}[i)]
		\item fast decaying if there exists $c > 0$ such that $\lim\limits_{r\rightarrow \infty} r^{\tilde{N}_- -2} u(r) = c$.
		\item slow decaying if there exists $c > 0$ such that $\lim\limits_{r\rightarrow \infty} r^{\alpha} u(r) = c$.
		\item pseudo slow decaying if there exist constants $0 < c_1 < c_2 $ such that $$c_1 = \liminf\limits_{r\rightarrow \infty} r^{\alpha} u(r) < \limsup\limits_{r\rightarrow \infty} r^{\alpha} u(r) = c_2 $$.
	\end{enumerate}
\end{Definition}

\noindent As shown in \cite{MNP} in terms of the dynamical system (\ref{system}), a fast decaying, slow decaying or pseudo-slow decaying solution corresponds to an orbit $\Gamma_p$ such that $\omega(\Gamma_p)=$ $A_0$, or $M_0$ or a periodic orbit around $M_0$ respectively. 

All the proprieties of the number $p_-^*$ stated in Theorem \ref{criticalexp} can be proved in the same way as in \cite{MNP}. We give some indications for the reader's convenience.

\begin{proof} [Proof of Theorem \ref{criticalexp}]
	It is not difficult to see that $p_-^+ \geq \frac{\tilde{N}_- +2}{\tilde{N}_- -2}$ because we already know that for $ p \leq \frac{\tilde{N}_- }{\tilde{N}_- -2}$ a solution of (\ref{probmenos}) in $B_R$ exists. Moreover for $p \in (\frac{\tilde{N}_- }{\tilde{N}_- -2}, \frac{\tilde{N}_- +2}{\tilde{N}_- -2})$, $M_0$ is a source, so if the trajectory $\Gamma_p$ starting from $N_0$ does not blow up in finite time, producing so a solution in the ball, $\omega(\Gamma_p)$ should be the stationary point $A_0$, because no periodic orbits exist, by Theorem \ref{porbits}. However $\omega(\Gamma_p)$ cannot be $A_0$, otherwise, the region bounded by $\Gamma_p$ and the $X$ and $Z$ axis would be an invariant region containing the point $M_0$ and the trajectories starting at $M_0$ would have no limit points to reach in forward time.
	\par The proof that $p^*_- \neq \frac{\tilde{N}_- +2}{\tilde{N}_- -2}$ is more delicate and can be carried out as in \cite{MNP} (see Theorem 5.2 and Theorem 4.7 therein). The fact that $p_-^* \leq \tilde{p}_-$ comes from Theorem \ref{soluball}, therefore using Theorem \ref{boundpmenos} we complete the estimate (\ref{expbound}).  
	\par The proprieties iv) and v) derive from the definitions of $\tilde{p}_-$ and $p_- ^*$. In particular we stress that for $p > \tilde{p}_-$ there are no periodic orbits so the only possibility is that the regular trajectory $\Gamma_p$ starting at $N_0$ converges to $M_0$ which means that the corresponding solution $u_p$ of (\ref{ODE}) is slow decaying.
	\par Finally i), ii) and iii) come from the proprieties of the critical exponent $p_-^*$, which can be proved exactly as in \cite{MNP}, see Theorem 5.2, Theorem 4.7, Proposition 4.2 and Corollary 4.4.
\end{proof}
\printbibliography

@Article{CP,
author={Pucci, Carlo},
title={Operatori ellittici estremanti},
journal={Annali di Matematica Pura ed Applicata (1923 -)},
year={1966},
}

@book{GT,
	title={Elliptic Partial Differential Equations of Second Order},
	author={Gilbarg, D. and Trudinger, N.S.},
	series={Classics in Mathematics},
	year={2015},
	publisher={Springer Berlin Heidelberg}
}

@Article{MNP,
	title={A dynamical system approach to a class of radial weighted fully nonlinear equations}, 
	author={Liliane Maia and Gabrielle Nornberg and Filomena Pacella},
	year={2020},
	eprint={2006.13093},
	archivePrefix={arXiv},
	primaryClass={math.AP}
}

@article {CL,
	AUTHOR = {Chen, Wen Xiong and Li, Congming},
	TITLE = {Classification of solutions of some nonlinear elliptic
	equations},
	JOURNAL = {Duke Math. J.},
	FJOURNAL = {Duke Mathematical Journal},
	VOLUME = {63},
	YEAR = {1991},
}

@article {CL1,
	AUTHOR = {Cutri , Alessandra and Leoni, Fabiana},
	TITLE = {On the {L}iouville property for fully nonlinear equations},
	JOURNAL = {Ann. Inst. H. Poincar\'{e} Anal. Non Lin\'{e}aire},
	FJOURNAL = {Annales de l'Institut Henri Poincar\'{e}. Analyse Non Lin\'{e}aire},
	VOLUME = {17},
	YEAR = {2000},
}

@article {QS,
	AUTHOR = {Quaas, Alexander and Sirakov, Boyan},
	TITLE = {Existence results for nonproper elliptic equations involving
	the {P}ucci operator},
	JOURNAL = {Comm. Partial Differential Equations},
	FJOURNAL = {Communications in Partial Differential Equations},
	VOLUME = {31},
	YEAR = {2006},
}

@article {FQ,
	AUTHOR = {Felmer, Patricio L. and Quaas, Alexander},
	TITLE = {On critical exponents for the {P}ucci's extremal operators},
	JOURNAL = {Ann. Inst. H. Poincar\'{e} Anal. Non Lin\'{e}aire},
	FJOURNAL = {Annales de l'Institut Henri Poincar\'{e}. Analyse Non Lin\'{e}aire},
	VOLUME = {20},
	YEAR = {2003},
	NUMBER = {5},
	}

@article {FQ1,
	AUTHOR = {Felmer, Patricio L. and Quaas, Alexander},
	TITLE = {Positive radial solutions to a semilinear equation involving
	the {P}ucci's operator},
	JOURNAL = {J. Differential Equations},
	FJOURNAL = {Journal of Differential Equations},
	VOLUME = {199},
	YEAR = {2004},
}

@article {AG,
	AUTHOR = {Adimurthi and Grossi, Massimo},
	TITLE = {Asymptotic estimates for a two-dimensional problem with
	polynomial nonlinearity},
	JOURNAL = {Proc. Amer. Math. Soc.},
	FJOURNAL = {Proceedings of the American Mathematical Society},
	VOLUME = {132},
	YEAR = {2004},
}

@article {MIP,
	AUTHOR = {De Marchis, Francesca and Ianni, Isabella and Pacella,
	Filomena},
	TITLE = {Asymptotic profile of positive solutions of {L}ane-{E}mden
	problems in dimension two},
	JOURNAL = {J. Fixed Point Theory Appl.},
	FJOURNAL = {Journal of Fixed Point Theory and Applications},
	VOLUME = {19},
	YEAR = {2017},
}

@article {LS,
	AUTHOR = {Da Lio, Francesca and Sirakov, Boyan},
	TITLE = {Symmetry results for viscosity solutions of fully nonlinear
	uniformly elliptic equations},
	JOURNAL = {J. Eur. Math. Soc. (JEMS)},
	FJOURNAL = {Journal of the European Mathematical Society (JEMS)},
	VOLUME = {9},
	YEAR = {2007},
}

@article {BGLP,
	AUTHOR = {Birindelli, Isabeau and Galise, Giulio and Leoni, Fabiana and
	Pacella, Filomena},
	TITLE = {Concentration and energy invariance for a class of fully
	nonlinear elliptic equations},
	JOURNAL = {Calc. Var. Partial Differential Equations},
	FJOURNAL = {Calculus of Variations and Partial Differential Equations},
	VOLUME = {57},
	YEAR = {2018},
}

@article {DF,
	AUTHOR = {Dancer, E. N. and Farina, Alberto},
	TITLE = {On the classification of solutions of {$-\Delta u=e^u$} on
	{$\Bbb R^N$}: stability outside a compact set and
	applications},
	JOURNAL = {Proc. Amer. Math. Soc.},
	FJOURNAL = {Proceedings of the American Mathematical Society},
	VOLUME = {137},
	YEAR = {2009},
}

@article {JL,
	AUTHOR = {Joseph, D. D. and Lundgren, T. S.},
	TITLE = {Quasilinear {D}irichlet problems driven by positive sources},
	JOURNAL = {Arch. Rational Mech. Anal.},
	FJOURNAL = {Archive for Rational Mechanics and Analysis},
	VOLUME = {49},
	YEAR = {1972/73},
}

\end{document}